\documentclass[11pt,reqno]{article}

\usepackage{graphicx,amsmath,amssymb,setspace,color,xcolor}
\usepackage{srcltx}
\usepackage{amsthm}

\newcommand{\be}{\begin{eqnarray}}
\newcommand{\ee}{\end{eqnarray}}
\newcommand{\by}{\begin{eqnarray*}}
\newcommand{\ey}{\end{eqnarray*}}
\newcommand{\bn}{\begin{enumerate}}
\newcommand{\en}{\end{enumerate}}
\newcommand{\bi}{\begin{itemize}}
\newcommand{\ei}{\end{itemize}}

\newtheorem{lemma}{Lemma}[section]

\newtheorem{remark}{Remark}[section]

\newtheorem{proposition}{Proposition}[section]
\newtheorem{example}{Example}[section]
\newtheorem{thm}{Theorem}[section]

\newcommand \bet {\beta}
\newcommand \la {\lambda}
\newcommand \tet {\theta}

\newcommand \alp {\alpha}
\newcommand \gam {\gamma}
\newcommand \eps {\varepsilon}
\newcommand \del {\delta}
\newcommand \upsn {\upsilon_n}
\newcommand \elln {\ell_n}
\newcommand \ome {\omega}

\newcommand \R {\mathbb{R}}
\newcommand \N {\mathbb{N}}
\newcommand \E {\mathbb{E}}
\newcommand \pr {\mathbb{P}}

\newcommand \lan {\la_n}
\newcommand \psin {\psi_n}
\newcommand \tetn {\tet_n}
\newcommand \cn {c_n}
\newcommand \Yn {Y_n}
\newcommand \Fn {F_n}
\newcommand \Rn {R_n}
\newcommand \rn {r_n}
\newcommand \An {A_n}

\newcommand \psid {\psi_D}

\newcommand \mO {\mathcal{O}}

\newcommand \sqrtn {\sqrt{n}}

\renewcommand{\theequation}{\arabic{section}.\arabic{equation}}

\numberwithin{equation}{section}

\begin{document}

\title{Rate of Convergence of the Probability of Ruin in the Cram\'er-Lundberg Model to its Diffusion Approximation\thanks{This is the final version of the paper. Appears on {\it Insurance: Mathematics and Economics}, Volume 93, July, 2020, 333--340.}}


\author{Asaf Cohen\footnote{Corresponding author. Department of Mathematics, University of Michigan, Ann Arbor, Michigan 48109, United States, Email: shloshim@gmail.com, web:https://sites.google.com/site/asafcohentau/
}
\and
Virginia R. Young\footnote{Department of Mathematics, University of Michigan, Ann Arbor, Michigan 48109, United States, \hfill \break vryoung@umich.edu.  V. R. Young thanks the Cecil J. and Ethel M. Nesbitt Chair of Actuarial Mathematics for partial financial support.  
}}

\maketitle

\begin{abstract}
We analyze the probability of ruin for the {\it scaled} classical Cram\'er-Lundberg (CL) risk process and the corresponding diffusion approximation.  The scaling, introduced by Iglehart \cite{I1969} to the actuarial literature, amounts to multiplying the Poisson rate $\la$ by $n$, dividing the claim severity by $\sqrtn$, and adjusting the premium rate so that net premium income remains constant.  

We are the first to use a comparison method to prove convergence of the probability of ruin for the scaled CL process and to derive the rate of convergence.  Specifically, we prove a comparison lemma for the corresponding integro-differential equation and use this comparison lemma to prove that the probability of ruin for the scaled CL process converges to the probability of ruin for the limiting diffusion process.  Moreover, we show that the rate of convergence for the ruin probability is of order $\mO\big(n^{-1/2}\big)$, and we show that the convergence is {\it uniform} with respect to the surplus.  To the best of our knowledge, this is the first rate of convergence achieved for these ruin probabilities, and we show that it is the tightest one in the general case.  For the case of exponentially-distributed claims, we are able to improve the approximation arising from the diffusion, attaining a uniform $\mO\big(n^{-k/2}\big)$ rate of convergence for arbitrary $k \in \N$.   We also include two examples that illustrate our results.

\medskip


\medskip


\medskip


\end{abstract}


\section{Introduction}

\setcounter{equation}{0}
\renewcommand{\theequation}
{1.\arabic{equation}}

Approximating and bounding the probability of ruin has a long-standing history in risk theory.  Arguably, the earliest approximation and bound are the well known {\it Cram\'er-Lundberg approximation} and related {\it Lundberg bound}; see Lundberg \cite{L1926} and Cram\'er \cite{C1930}.

The Cram\'er-Lundberg approximation is asymptotic for large values of the surplus process, and in most of the literature in ruin theory, this is what asymptotic refers to.  In this paper, asymptotic refers to the large values of the Poisson rate, together with small claim severity; in that case, the classical Cram\'er-Lundberg risk process approaches a diffusion.  We combine two areas of research, which we discuss below.

The first area is that of using a diffusion process to approximate the discrete risk process, with many small jumps. This method takes advantage of the mathematical tractability of diffusion processes to deduce properties of the process it approximates. This technique was introduced by Kingman \cite{Kingman} in the analysis of a single-server queue and by Iglehart and Whitt \cite{IW1970a, IW1970b} in the context of multiple channel queues.  Since then, it has gained popularity in the stochastic networks community, where it is referred to as the {\it heavy-traffic approximation}.  In this field, the length of the queues are scaled (divided) by $n^{1/2}$ and the rates are scaled (multiplied) by $n$ in such a way that the system is {\it critically loaded} in the sense that the traffic intensity (utilization) converges to $1$ from below. The martingale/functional central limit theorem, then, implies that, in the limit, one attains a diffusion process. The approximation helps in finding asymptotic optimal controls and behavior of complicated systems.  For a basic introduction to the heavy-traffic approximation, please see Chen and Yao \cite{CY2001}, Kushner \cite{K2001}, and the references therein.

Iglehart \cite{I1969} introduced the diffusion approximation to the actuarial literature. He used probabilistic techniques (weak convergence) to prove that the probability of ruin for the scaled model approaches the probability of ruin for the limiting diffusion process.  Grandell \cite{G1977} and Asmussen \cite{A1984} further used the approximating diffusion process to approximate the probability of ruin in finite time; Asmussen's work was inspired by Siegmund \cite{S1979}.  In these works, the limits hold pointwise and no rate of convergence is provided.  More recently, B\"auerle \cite{B2004} used probabilistic techniques to prove limiting results under optimal control of the surplus process.

Instead of probabilistic techniques, we rely on comparison analysis of the integro-differential equation that the probability of ruin solves, and this is the second area of research.  The key element of this technique is an ``increasing'' functional that vanishes when evaluated at the probability of ruin (in the $n$-scaled problem). By perturbing the probability of ruin by $\mO\big(n^{-1/2}\big)$ in both directions and by using the monotonicity of the functional, we get the required bounds (Propositions \ref{prop:psid_scale} and \ref{prop:upsn}). This in turn implies a rate of convergence of order $\mO\big(n^{-1/2}\big)$, uniformly in the initial surplus (Theorem \ref{thm:psin_lim}).  In fact, due to the $\mO\big(n^{-1/2}\big)$ jump sizes, it is the best rate that can be achieved for the general case (Remark \ref{rem:opt}).  Moreover, this is the first time that a comparison principle is used to obtain the rate of convergence of these ruin probabilities.  We believe that this technique can be applied in other actuarial and queueing applications, which we detail in Section \ref{sec:sum}.

Several actuarial researchers used comparison to bound the probability of ruin; however, they worked in the primary problem ($n = 1$) and did not apply comparison to the $n$-scaled problem.  Specifically, Taylor \cite{T1976} bounded the probability of ruin by using the integral version of this equation and comparison results for Volterra integral operators; see Walter \cite{W1970} for these comparison results.  De Vylder and Goovaerts \cite{dVG1984} and Broeckx, De Vylder, and Goovaerts \cite{BdVG1986} continued the work of Taylor \cite{T1976}, using a simpler comparison lemma.

The remainder of the paper is organized as follows.  In Section \ref{sec:back}, we present the Cram\'er-Lundberg model and prove a comparison lemma for the integro-differential equation that determines the probability of ruin in that model.  In Section \ref{sec:asymp}, we scale the model by $n$ and remind the reader of the probability of ruin in the diffusion approximation.  In Section \ref{sec:psid}, we prove that the probability of ruin in the scaled model approaches $\psid$ at a rate of convergence of order $\mO( n^{-1/2})$, uniformly in the initial surplus.  We also strengthen this result for the special case of exponentially distributed claims and show that we cannot strengthen the rate of convergence more generally.  Section \ref{sec:sum} concludes our paper.

\section{Classical risk model and comparison lemma}\label{sec:back}

\setcounter{equation}{0}
\renewcommand{\theequation}
{2.\arabic{equation}}

\subsection{Cram\'er-Lundberg model}\label{sec:model}

Consider an insurer whose surplus process $X = \{ X_t \}_{t \ge 0}$ is described by the classical Cram\'er-Lundberg model, that is, the insurer receives premium income at a constant rate $c$ and pays claims according to a compound Poisson process.  Specifically,
\begin{equation}\label{eq:X}
X_t = x + ct - \sum_{i=1}^{N_t} Y_i,
\end{equation}
in which $X_0 = x \ge 0$ is the initial surplus, $N = \{N_t \}_{t \ge 0}$ is a homogeneous Poisson process with intensity $\la > 0$, and the claim sizes $Y_1, Y_2, \ldots$ are independent and identically distributed, positive random variables, independent of $N$. Let $F_Y$ denote the common cumulative distribution function of $\{Y_i\}_{i \in \N}$. Assume that $Y$ has finite moment generating function $M_Y(u) = \E \big(e^{Yu} \big)$ for $u$ in a neighborhood of $0$, say, for $u \in (-u_0, u_0)$ for some $u_0 > 0$; thus, $\E \big( Y^k \big) < \infty$ for $k = 1, 2, \dots$.   Finally, assume that the premium rate $c$ satisfies $c > \la \E Y$ (otherwise, eventual ruin is certain), and write $c = (1 + \tet) \la \E Y$, with positive risk loading $\tet > 0$.

Define the time of ruin $\tau$ by
\begin{equation}\label{eq:tau}
\tau = \inf\{ t\ge 0: X_t < 0\},
\end{equation}
and define the probability of ultimate ruin by
\begin{equation}\label{eq:psi}
\psi(x) = \pr \big(\tau < \infty \mid X_0 = x \big).
\end{equation}
Recall that $\psi(0) = 1/(1 + \tet)$.  Standard risk theory texts\footnote{See, for example, Section 5.3 in the recent text by Schmidli \cite{S2017}.} demonstrate that one can characterize $\psi$ as the unique classical solution of the following integro-differential equation on $\R^+$ subject to a boundary condition at infinity:
\begin{equation}\label{eq:IDeq}
\begin{cases}
\la v(x) = c v_x(x+) + \la \displaystyle \int_0^x v(x - y) dF_Y(y) + \la S_Y(x), \quad x > 0, \\
\lim \limits_{x \to \infty} v(x) = 0,
\end{cases}
\end{equation}
in which $S_Y = 1 - F_Y$ is $Y$'s survival function.  If we substitute for $c = (1 + \tet) \la \E Y$ in \eqref{eq:IDeq}, then we see that $\psi$ is independent of $\la$.

\begin{remark}
In Section $5.3$, Schmidli {\rm \cite{S2017}} showed that one can rewrite the differential equation in \eqref{eq:IDeq} as an integral equation, as follows:
\begin{equation}\label{eq:Int}
c v(x) = \la \int_0^x v(x - y) S_Y(y) dy + \la \int_x^\infty S_Y(y) dy.
\end{equation}
It is this form of the equation that Taylor {\rm \cite{T1976}}, DeVylder and Goovaerts {\rm \cite{dVG1984}}, and Broeckx, DeVylder, and Goovaerts {\rm \cite{BdVG1986}} used to find bounds for the probability of ruin.  Theorem $1.2.1$ in De Vylder and Goovaerts {\rm \cite{dVG1984}} proves that \eqref{eq:Int} has a unique solution; thus, \eqref{eq:IDeq} has a unique solution, and it equals the probability of ruin.

In future work, we will control the surplus process $X$, and an integral equation of the form \eqref{eq:Int} will not readily apply in that case.  Thus, anticipating that future work, we continue with the integro-differential equation in \eqref{eq:IDeq}.   \qed
\end{remark}

\subsection{Comparison lemma}\label{sec:comp_lem}

We look for bounds for the probability of ruin $\psi$ as sub- and super-solutions of \eqref{eq:IDeq}.  Thus, we begin by proving a comparison lemma, which we use to determine whether a given function is a lower or upper bound for $\psi$.

Define the operator $F$ by
\begin{equation}\label{eq:F}
F \big(x, u(x), u_x(x+), u(\cdot) \big) = - c u_x(x+) - \la \left( \int_0^x u(x - y) dF_Y(y) + S_Y(x) - u(x) \right).
\end{equation}
Note that the probability of ruin $\psi$ satisfies $F \big(x, \psi(x), \psi_x(x+), \psi(\cdot) \big) = 0$ for all $x > 0$.  We expect the probability of ruin $\psi$ to be continuous on all of $\R^+$ and to have continuous first derivatives on $\R^+$ except at points of discontinuity of $F_Y$; thus, we require $u$ and $v$ in Lemma \ref{lem:comp} below to satisfy similar continuity properties.

\begin{lemma}\label{lem:comp}
$($Comparison lemma$)$.  Let $0 \le a < b \le \infty$, and consider functions $u, v \in \mathcal{C} \big([0, b) \big)$ with continuous first derivatives, except possibly at points of discontinuity of $F_Y$, where $u$ and $v$ have left- and right-derivatives.  Suppose $u$ and $v$ are such that $u(x) \le v(x)$ for all $0 \le x \le a$ and $u(b) \le v(b)$.\footnote{If $b = \infty$, then $u(b)$ denotes $\lim_{x \to \infty} u(x)$; similarly, for $v(b)$.}  Furthermore, suppose
\[
F \big(x, u(x), u_x(x+), u(\cdot) \big) < F \big(x, v(x), v_x(x+), v(\cdot) \big), 
\]
for all $x \in (a, b)$, then $u(x) < v(x)$ for all $x \in(a, b)$.
\end{lemma}

\begin{proof} First, if the maximum of $u - v$ on $[a, b]$ occurs at $x = a$ or $x = b$, but not in the interior of $(a, b)$, then $u < v$ in the interior  because $u - v \le 0$ on the boundary, by assumption.

Second, if $u - v$ attains a strictly negative maximum in the interior of $(a, b)$, then we also have $u < v$ in the interior.

Third, if $u - v$ attains a non-negative maximum at $x_0 \in (a, b)$, then $u_x(x_0+) - v_x(x_0+) \le 0$.  It follows that
\begin{align*}
0 &< F \big(x_0, v(x_0), v_x(x_0+), v(\cdot) \big) - F \big(x_0, u(x_0), u_x(x_0+), u(\cdot) \big) \\
&= - c v_x(x_0+) - \la \int_0^{x_0} v(x_0 - y) dF_Y(y) - \la S_Y(x_0) + \la v(x_0)  \\
&\quad + c u_x(x_0+) + \la \int_0^{x_0} u(x_0 - y) dF_Y(y) + \la S_Y(x_0) - \la u(x_0) \\
&\le \la \int_0^{x_0} \big( u(x_0 - y) - v(x_0 - y) \big) dF_Y(y) - \la \big( u(x_0) - v(x_0) \big).
\end{align*}
The last line is non-positive.  Indeed, because $u - v$ reaches a non-negative maximum at $x = x_0 \in (a, b)$, we have $u(x_0) - v(x_0) \ge u(x) - v(x)$ for all $x \in (a, b)$.  Furthermore, because $u(x) \le v(x)$ for all $0 \le x \le a$, we have $u(x_0) - v(x_0) \ge u(x) - v(x)$ for all $x \in (0, b)$.  Without loss of generality, we can extend $u$ and $v$ into $\R^-$ by setting $u(x) = v(x) = 1$ for $x < 0$, from which it follows that $u(x_0) - v(x_0) \ge u(x) - v(x)$ for all $x < b$.  We deduce that $u(x) - u(x_0) \le v(x) - v(x_0)$ for all $x \le x_0$, which implies
\begin{align*}
0 &< \la \int_0^{x_0} \big( u(x_0 - y) - v(x_0 - y) \big) dF_Y(y) - \la \big( u(x_0) - v(x_0) \big) \\
&= \la \int_0^\infty \Big( \big( u(x_0 - y) - u(x_0) \big) - \big( v(x_0 - y) - v(x_0) \big) \Big) dF_Y(y) \le 0,
\end{align*}
a contradiction.  Thus, $u < v$ in $(a, b)$.
\end{proof}

\begin{remark}\label{rem:non-strict}
If we only want non-strict comparison, that is, $u \le v$, then we can weaken the sub-$($super-$)$solution property to $F \big(x, u(x), u_x(x+), u(\cdot) \big) \le F \big(x, v(x), v_x(x+), v(\cdot) \big)$, with \hfill \break $F \big(x, u(x), u_x(x+), u(\cdot) \big)$ finite.  \qed
\end{remark}

In the next example, we use Lemma \ref{lem:comp} and Remark \ref{rem:non-strict} to re-prove the well known Lundberg bound.

\begin{example}
Suppose $R > 0$ solves
\[
c R = \la \big(M_Y(R) - 1 \big),
\]
that is, $R$ is the {\rm adjustment coefficient}, and define $v(x) = e^{-R x}$.  Let $a = 0$ and $b = \infty$ in Lemma {\rm \ref{lem:comp};} then, $\psi(0) = 1/(1 + \tet) \le 1 = v(0)$.  Also, recall, for $x > 0,$ $F \big(x, \psi(x), \psi_x(x+), \psi(\cdot) \big) = 0,$ and
\begin{align*}
F \big(x, v(x), v_x(x), v(\cdot) \big) &=  c R e^{-Rx} - \la \left( \int_0^x e^{-R(x - y)} dF_Y(y) + S_Y(x) - e^{-Rx} \right) \\
&= \la e^{-Rx} \left[ \big(M_Y(R) - 1 \big) - \left( \int_0^x e^{Ry} dF_Y(y) + e^{Rx} S_Y(x) - 1 \right) \right] \\
&= \la e^{-Rx} \int_x^\infty \big(e^{Ry} - e^{Rx} \big) dF_Y(y) \ge 0.
\end{align*}
We deduce from the non-strict version of Lemma {\rm \ref{lem:comp}} that $\psi(x) \le e^{-Rx}$ for $x > 0$, the Lundberg bound.  \qed
\end{example}

\section{Scaled model and diffusion approximation}\label{sec:asymp}

\setcounter{equation}{0}
\renewcommand{\theequation}
{3.\arabic{equation}}

Next, we scale our model by $n > 0$.  In the scaled system, define $\lan = n \la$, so $n$ large is essentially equivalent to $\la$ large. Scale the claim severity by defining $\Yn = Y/\sqrtn$; thus, the variance of total claims during $[0, t]$ is invariant under the scaling, that is, $\lan \E \big(\Yn^2 \big) = \la \E \big(Y^2 \big)$ for all $n > 0$.  Finally, define the premium rate by $\cn = c + (\sqrtn - 1) \la \E Y$; thus, $\cn - \lan \E \Yn = c - \la \E Y$ is also invariant under the scaling.  We can also write $\cn = (\sqrtn + \tet) \la \E Y$, in which $c = (1 + \tet) \la \E Y$; moreover, we can write $\cn = (1 + \tetn) \lan \E \Yn$, in which $\tetn = \tet/\sqrtn$.  The diffusion approximation of the scaled surplus process is, therefore, 
\begin{equation}\label{eq:diff_approx}
\big( \cn - \lan \E \Yn \big)dt + \sqrt{\lan \E \big(\Yn^2 \big)} \, dB_t = \big( c - \la \E Y \big)dt + \sqrt{\la \E \big(Y^2 \big)} \, dB_t,
\end{equation}
for some standard Brownian motion $B = \{ B_t \}_{t \ge 0}$, independent of $n$.  See Iglehart \cite{I1969}, B\"auerle \cite{B2004}, Gerber, Shiu, and Smith \cite{GSS2008}, and Schmidli \cite{S2017} for more information about this scaling.

Let $\psid$ denote the probability of ruin for the diffusion approximation; then, $\psid$ uniquely solves the following boundary-value problem:
\begin{equation}\label{eq:BVP0}
\begin{cases}
0 = \tet \E Y v_x(x) + \dfrac{1}{2} \, \E \big(Y^2\big) v_{xx}(x), \quad x > 0, \\
v(0) = 1, \qquad \lim \limits_{x \to \infty} v(x) = 0.
\end{cases}
\end{equation}
Because $c - \la \E Y = \la \tet \E Y$, we were able to eliminate a factor of $\la$ to obtain the ode in \eqref{eq:BVP0}; thus, we see that $\psid$ is independent of $\la$, as is the probability of ruin in the Cram\'er-Lundberg model $\psi$.  The solution of \eqref{eq:BVP0} is given by
\begin{equation}\label{eq:psid}
\psid(x) = e^{-\gam x},
\end{equation}
for all $x > 0$, in which $\gam$ equals
\begin{equation}\label{eq:gam}
\gam = \dfrac{2 \tet \E Y}{\E \big( Y^2 \big)} \, .
\end{equation}
For an early reference of \eqref{eq:psid}, see Theorem 8 in Iglehart \cite{I1969}.  Because the diffusion in \eqref{eq:diff_approx} approximates the Cram\'er-Lundberg risk process in \eqref{eq:X} with $\la$, $Y$, and $c$ replaced by $\lan$, $\Yn$, and $\cn$, respectively, researchers often say that $\psid$ approximates $\psin$.  In Theorem {\rm \ref{thm:psin_lim}} in the next section, we quantify the degree to which $\psid$ approximates $\psin$.

We end this section with two examples, in which we (attempt to) expand $\psin$ in powers of order $\mO \big( n^{-1/2} \big)$.   In the first example, we compute this expansion when $Y$ is distributed exponentially.

\begin{example}\label{ex:exp}
Suppose $Y \sim Exp(\bet)$ with mean $1/\bet;$ then, $\Yn \sim Exp(\sqrtn \bet)$ and
\begin{align} \label{eq:exp_psin}
\psin(x) &= \dfrac{1}{1 + \tetn} \, \exp \bigg( - \, \dfrac{\tetn x}{(1 + \tetn) \E \Yn} \bigg) = \dfrac{1}{1 + \frac{\tet}{\sqrtn}} \, \exp \Bigg( - \, \dfrac{\tet \bet x}{1 + \frac{\tet}{\sqrtn}} \Bigg).
\end{align}
To obtain the $\mO(n^0)$ term in $\psin$, compute
\[
\lim_{n \to \infty} \psin(x) = e^{-\tet \bet x},
\]
and note that $\gam$ defined in \eqref{eq:gam} equals $\tet \bet$ in this example.  Thus, the probability of ruin for the diffusion approximation $\psid$ is the leading-order term in the expansion of $\psin$.  Next, to obtain the $\mO\big(n^{-1/2} \big)$ term, compute
\[
\lim_{n \to \infty} \sqrtn \left( \psin(x) - e^{-\tet \bet x} \right) = \tet \big( \tet \bet x - 1 \big) e^{- \tet \bet x}.
\]
Thus, we have the following expansion of $\psin$ for $x \ge 0:$ 
\begin{align}
\psin(x) = e^{-\tet \bet x} \left\{ 1 + \dfrac{\tet}{\sqrtn} \, \big( \tet \bet x - 1 \big) \right\} + \mO \big( n^{-1} \big).
\label{eq:exp_asymp}
\end{align}
In Theorem {\rm \ref{thm:exp}} in the next section, we show that this expansion is valid uniformly with respect to $x \ge 0$.  More strongly, we show that we can extend the expansion to order $\mO \big( n^{-k/2} \big)$ $($uniformly in $x)$ for any $k \in \N$.
\qed
\end{example}

In the second example, we compute the same expansion when $Y$ is distributed according to the Gamma with shape parameter $2$.

\begin{example}\label{ex:gamma}
Suppose $Y \sim Gamma(2, \bet)$ with mean $2/\bet;$ then, $\Yn \sim Gamma(2, \sqrtn \bet)$ and
\begin{align}\label{eq:gam_psin}
\psin(x) &= \dfrac{\tetn}{2(1 + \tetn)} \left\{ \dfrac{2 \sqrtn \bet - \Rn}{2 \tetn \cdot \sqrtn \bet - \frac{3 + 4 \tetn}{2} \, \Rn} \, e^{- \Rn x} + \dfrac{2 \sqrtn \bet - \rn}{2 \tetn \cdot \sqrtn \bet - \frac{3 + 4\tetn}{2} \, \rn} \, e^{- \rn x}  \right\} \notag \\
&= \dfrac{\tet}{2 \Big(1 + \frac{\tet}{\sqrtn} \Big)} \left\{ \dfrac{2 \bet - \frac{\Rn}{\sqrtn}}{2 \tet \bet - \frac{3 + 4 \tet/\sqrtn}{2} \, \Rn} \, e^{- \Rn x} + \dfrac{2 \bet - \frac{\rn}{\sqrtn}}{2 \tet \bet - \frac{3 + 4 \tet/\sqrtn}{2} \, \rn} \, e^{- \rn x}  \right\},
\end{align}
in which $\Rn$ is the adjustment coefficient
\[
\Rn =  \dfrac{\sqrtn \bet}{4 \Big(1 + \frac{\tet}{\sqrtn} \Big)} \left[ \left( 3 + \dfrac{4 \tet}{\sqrtn}\right) - \sqrt{9 + \dfrac{8 \tet}{\sqrtn}} \, \right],
\]
and $\rn$ is
\[
\rn = \dfrac{\sqrtn \bet}{4 \Big(1 + \frac{\tet}{\sqrtn} \Big)} \left[ \left( 3 + \dfrac{4 \tet}{\sqrtn}\right) + \sqrt{9 + \dfrac{8 \tet}{\sqrtn}} \, \right].
\]
One can show that $\lim_{n \to \infty} \Rn = \frac{2}{3} \, \tet \bet = \gam$,
\[
\lim_{n \to \infty} \psin(x) = e^{- \, \frac{2}{3} \, \tet \bet x},
\]
and
\[
\lim_{n \to \infty} \sqrtn \left( \psin(x) - e^{- \, \frac{2}{3} \, \tet \bet x} \right) = 
\begin{cases}
\dfrac{8 \tet}{9} \left( \dfrac{2}{3} \, \tet \bet x - 1 \right) e^{- \, \frac{2}{3} \, \tet \bet x}, &\quad x > 0, \\
- \tet, &\quad x = 0.
\end{cases}
\]
The corresponding expansion equals
\begin{equation}\label{eq:gam_GSS}
e^{- \, \frac{2}{3} \, \tet \bet x} \left\{1 + \dfrac{8\tet}{9\sqrtn} \left( \dfrac{2}{3} \, \tet \bet x - 1 \right) \right\} + \mO \big( n^{-1} \big),
\end{equation}
for $x > 0$, and equals
\[
1 - \dfrac{\tet}{\sqrtn} + \mO \big( n^{-1} \big),
\]
for $x = 0$, in which the $\mO \big( n^{-1} \big)$ term depends on $x$ such that, for any given $x$, there is a $C_x > 0$ with $\big| \mO\big(n^{-1}\big) \big| \le C_x \, n^{-1}$.  Note that the expansion is discontinuous at $x = 0$, which implies that this expansion is not valid in some sense, which we will discuss further in Example {\rm \ref{ex:gamma2}} below.  \qed
\end{example}

\section{Asymptotic analysis}\label{sec:psid}

\setcounter{equation}{0}
\renewcommand{\theequation}
{4.\arabic{equation}}

In this section, we use Lemma \ref{lem:comp}, as it applies to the scaled problem, to show that $\psid$ from \eqref{eq:psid} approximates the probability of ruin $\psin$ for the scaled problem to order $\mO \big( n^{-1/2} \big)$ uniformly with respect to $x$. 

Throughout this section, let $\Fn$ denote the operator in \eqref{eq:F} with $c$, $\la$, and $Y$ replaced by $\cn$, $\lan$, and $\Yn$, respectively, and write $\Fn$ as follows:
\begin{align}\label{eq:Fn}
&\Fn\big(x, u(x), u_x(x+), u(\cdot) \big) = \notag \\
&\quad  -\la \left[ \big(\sqrtn + \tet \big) \E Y u_x(x+) + n \left( \int_0^{\sqrtn x} u\Big(x - \frac{t}{\sqrtn} \Big) dF_Y(t) + S_Y(\sqrtn x) - u(x) \right) \right].
\end{align}
Note that $\Fn \big(\psin, \psin(x), \psin'(x+), \psin(\cdot) \big) = 0$ for all $x > 0$, and $\psin$ is independent of $\la$ but clearly dependent on $n$.

In the next two propositions, we modify $\psid$ to obtain lower and upper bounds for $\psin$, respectively.  In Appendix \ref{app:psid}, we present background calculations that inspired these bounds.  We begin by modifying $\psid$ to obtain a lower bound for $\psin$.

\begin{proposition}\label{prop:psid_scale}
Suppose $m$ exists such that
\begin{equation}\label{eq:Zd}
\sup \limits_{d \ge 0} \E \Big( (Y - d)^2 \, e^{\frac{\gam}{\sqrt{m}} (Y - d)} \, \Big| \, Y > d \Big) < \infty.
\end{equation}
Choose $\eps > 0$, and define $\del$ by
\begin{equation}\label{eq:del}
\del = \max \left[ \tet, \, \sup \limits_{d \ge 0} \big( \gam \E(Y - d \, | \, Y> d) + \eps \big) \right],
\end{equation}
and choose $N > \max \big( \del^2, m \big)$ such that
\begin{equation}\label{eq:N}
\sup \limits_{d \ge 0} \dfrac{\gam^2}{\sqrt{N}} \int_0^1 (1 - \ome) \E\Big( (Y - d)^2 \, e^{\frac{\gam \ome}{\sqrt{N}} (Y - d)} \, \Big|  \, Y > d \Big) d \ome \le \eps.
\end{equation}
Then, for all $n > N$,
\begin{equation}\label{eq:psid_scale}
\left( 1 - \dfrac{\del}{\sqrtn} \right) \psid(x) < \psin(x),
\end{equation}
for all $x \ge 0$.
\end{proposition}

\begin{proof}
Because $\del \ge \tet$, we have
\[
\left( 1 - \dfrac{\del}{\sqrtn} \right) \psid(0) = 1 - \dfrac{\del}{\sqrtn} \le 1 - \dfrac{\tet}{\sqrtn} < \dfrac{1}{1 + \frac{\tet}{\sqrtn}} = \psin(0).
\]
Also, $\lim_{x \to \infty} \left( 1 - \frac{\del}{\sqrtn} \right) \psid(x) = 0 = \lim_{x \to \infty} \psin(x)$.

Next, consider $\Fn$ evaluated at $\elln$, in which $\elln(x) = \big(1 - \del/\sqrtn \, \big) \psid(x)$, and assume without loss of generality that $n > \del^2$:
\begin{align}\label{eq:Fn_del}
&\Fn \big( x, \elln(x), \elln'(x), \elln(\cdot) \big) \notag \\
&\quad =  \la \left( 1 - \dfrac{\del}{\sqrtn} \right) e^{- \gam x} \left\{ \big( \sqrtn + \tet \big) \E Y \gam - n \int_0^\infty \left( e^{\frac{\gam t}{\sqrtn}} - 1 \right) dF_Y(t) \right\} \notag \\
&\qquad + n \la \int_{\sqrtn x}^\infty \left( \left( 1 - \dfrac{\del}{\sqrtn} \right) e^{\gam \left(\frac{t}{\sqrtn} - x \right)} - 1 \right) dF_Y(t) \notag \\
&\quad \propto \big( \sqrtn + \tet \big) \E Y \gam - n \int_0^\infty \left( e^{\frac{\gam t}{\sqrtn}} - 1 \right) dF_Y(t) + n \int_{\sqrtn x}^\infty \left(  e^{ \frac{\gam t}{\sqrtn}} - \dfrac{1}{1 - \frac{\del}{\sqrtn}} \, e^{\gam x}\right) dF_Y(t) \notag \\
&\quad \propto  - \int_0^\infty \left( e^{\frac{\gam t}{\sqrtn}} - 1 - \dfrac{\gam t}{\sqrtn} - \dfrac{\gam^2 t^2}{2 n} \right) dF_Y(t) + \int_{\sqrtn x}^\infty \left(  e^{ \frac{\gam t}{\sqrtn}} - \dfrac{1}{1 - \frac{\del}{\sqrtn}} \, e^{\gam x}\right) dF_Y(t).
\end{align}  
The first integral is automatically negative; thus, if we find values of $\del$ and $N > \del^2$ for which the second integral is non-positive for all $n > N$ and for all $x \ge 0$, then Lemma \ref{lem:comp} implies that $\elln(x) < \psin(x)$ for all $x \ge 0$ and for all $n > N$.  To that end, consider the following inequality:
\[
\int_{\sqrtn x}^\infty \left(  e^{ \frac{\gam t}{\sqrtn}} - \dfrac{1}{1 - \frac{\del}{\sqrtn}} \, e^{\gam x}\right) dF_Y(t) \le 0.
\]
If $S_Y(\sqrtn x) = 0$, then the left side is identically $0$, so suppose that $S_Y(\sqrtn x) > 0$.  After replacing $\sqrtn x$ by $d$ and dividing by $e^{\gam x} S_Y(d)$, the above inequality becomes
\[
\int_d^\infty \left(  e^{ \frac{\gam}{\sqrtn}(t - d)} - \dfrac{1}{1 - \frac{\del}{\sqrtn}} \right) \dfrac{dF_Y(t)}{S_Y(d)} \le 0,
\]
for $d \ge 0$, or equivalently,
\[
\int_d^\infty \left(  e^{ \frac{\gam}{\sqrtn} (t - d)} - 1 \right)  \dfrac{dF_Y(t)}{S_Y(d)} \le \dfrac{\frac{\del}{\sqrtn}}{1 - \frac{\del}{\sqrtn}} \, .
\]
If we find $\del$ that satisfies the following stronger inequality, then the above sequence of inequalities holds:
\begin{equation}\label{ineq:del1}
\int_d^\infty \left(  e^{ \frac{\gam}{\sqrtn}(t - d)} - 1 \right) \dfrac{dF_Y(t)}{S_Y(d)} \le \dfrac{\del}{\sqrtn} \, .
\end{equation}
Rewrite the integrand from the left side of inequality \eqref{ineq:del1} as follows, with $z = t - d$:
\[
e^{ \frac{\gam z}{\sqrtn}} - 1 = \dfrac{\gam z}{\sqrtn} + \dfrac{\gam^2 z^2}{n} \int_0^1 (1 - \ome) e^{\frac{\gam z}{\sqrtn} \, \ome} d\ome.
\]
Thus, inequality \eqref{ineq:del1} is equivalent to
\[
\int_d^\infty \left(  \dfrac{\gam(t - d)}{\sqrtn} + \dfrac{\gam^2 (t - d)^2}{n} \int_0^1 (1 - \ome) e^{\frac{\gam(t - d)}{\sqrtn} \, \ome} d \ome \right) \dfrac{dF_Y(t)}{S_Y(d)} \le \dfrac{\del}{\sqrtn} \, ,
\]
or, after multiplying both side by $\sqrtn$ and switching the order of integration,
\begin{equation}\label{ineq:del2}
\gam \E(Y - d \, | \, Y > d) + \dfrac{\gam^2}{\sqrtn} \int_0^1 (1 - \ome) \, \E \Big( (Y - d)^2 \, e^{\frac{\gam \ome}{\sqrtn} (Y - d)} \, \Big| \, Y > d \Big) d \ome \le \del.
\end{equation}
Note that the left side decreases with increasing $n$.  It follows that if we define $\del$ and $N$ as in \eqref{eq:del} and \eqref{eq:N}, respectively, then inequality \eqref{ineq:del2} holds for all $d \ge 0$ and all $n > N$, which implies that $F_n$ evaluated at $\elln$ is negative for all $x \ge 0$ and all $n > N$.  The conclusion in \eqref{eq:psid_scale}, then, follows from Lemma \ref{lem:comp} because $\Fn$ evaluated at $\psin$ equals $0$.
\end{proof}

In the following proposition, we modify $\psid$ to obtain an upper bound for $\psin$.

\begin{proposition}\label{prop:upsn}
Define the function $\upsn$ by
\begin{equation}\label{eq:upsn}
\upsn(x) = e^{- \left(\gam - \frac{\alp}{\sqrtn} \right) x} = \psid(x) \, e^{\frac{\alp}{\sqrtn} \, x}.
\end{equation}
If
\begin{equation}\label{eq:alp}
\alp > \dfrac{\gam^2}{3} \, \dfrac{\E \big(Y^3 \big)}{\E \big(Y^2\big)} \, ,
\end{equation}
then there exists $N > 0$ such that, for all $n > N$,
\begin{equation}\label{eq:upsn_bnd}
\psin(x) < \upsn(x),
\end{equation}
for all $x \ge 0$.
\end{proposition}

\begin{proof}
$\psin(0) < 1 = \upsn(0)$.  Also, $\lim_{x \to \infty} \psin(x) = 0 = \lim_{x \to \infty} \upsn(x)$, if $n > (\alp/\gam)^2$.

Next, consider $\Fn$ evaluated at $\upsn$:
\begin{align}\label{eq:Fn_alp}
&\Fn \big( x, \upsn(x), \upsn'(x), \upsn(\cdot) \big) \notag \\
&\quad =  \la e^{- \left( \gam - \frac{\alp}{\sqrtn} \right) x} \left\{ \big( \sqrtn + \tet \big) \E Y \left( \gam - \dfrac{\alp}{\sqrtn} \right) - n \int_0^\infty \left( e^{ \left(\gam - \frac{\alp}{\sqrtn} \right) \frac{t}{\sqrtn}} - 1 \right) dF_Y(t) \right\} \notag \\
&\qquad + n \la \int_{\sqrtn x}^\infty \left( e^{ \left(\gam - \frac{\alp}{\sqrtn} \right) \left(\frac{t}{\sqrtn} - x \right)} - 1 \right) dF_Y(t).
\end{align}
The last line of \eqref{eq:Fn_alp} is automatically non-negative if $n > (\alp/\gam)^2$.  The expression in curly brackets is independent of $x$; denote it by $\An$.  If we find values of $\alp$ and $N > (\alp/\gam)^2$ for which $\An$ is positive for all $n > N$, then Lemma \ref{lem:comp} implies that $\psin(x) < \upsn(x)$ for all $x \ge 0$ and for all $n > N$.  Expand the exponential in the integrand in $\An$ to obtain
\begin{align*}
\An &= \big( \sqrtn + \tet \big) \E Y \left( \gam - \dfrac{\alp}{\sqrtn} \right) \\
&\quad  - n \int_0^\infty \left( e^{\left(\gam - \frac{\alp}{\sqrtn} \right) \frac{t}{\sqrtn}} - 1 - \left(\gam - \frac{\alp}{\sqrtn} \right) \frac{t}{\sqrtn} - \left(\gam - \frac{\alp}{\sqrtn} \right)^2 \dfrac{t^2}{2n} - \left(\gam - \frac{\alp}{\sqrtn} \right)^3 \dfrac{t^3}{6 n^{3/2}} \right) dF_Y(t) \\
&\quad - n \left( \left(\gam - \frac{\alp}{\sqrtn} \right) \dfrac{\E Y}{\sqrtn} + \left(\gam - \frac{\alp}{\sqrtn} \right)^2 \dfrac{\E \big(Y^2 \big)}{2n} + \left(\gam - \frac{\alp}{\sqrtn} \right)^3 \dfrac{\E \big(Y^3 \big)}{6 n^{3/2}} \right) \\
&= \dfrac{\gam}{2 \sqrtn} \left( \alp \E \big(Y^2 \big) - \dfrac{\gam^2}{3} \, \E \big(Y^3 \big) \right) + \mO \big( n^{-1} \big).
\end{align*}
Choose $\alp$ as in \eqref{eq:alp}; then, the first term in the above expression is strictly positive.  Next, choose $N > (\alp/\gam)^2$ so that the absolute value of the remainder term in $A_N$ (if it is negative) is less than the first term.  It follows that $\An > 0$ for that choice of $\alp$ and for all $n > N$.  The conclusion in \eqref{eq:upsn_bnd}, then, follows from Lemma \ref{lem:comp} because $\Fn$ evaluated at $\psin$ equals $0$ and $\Fn$ evaluated at $v_n$ is positive.
\end{proof}

In the following theorem, we combine the results of Propositions \ref{prop:psid_scale} and \ref{prop:upsn}.

\begin{thm}\label{thm:psin_lim}
If \eqref{eq:Zd} holds,  then there exist $C > 0$ and $N > 0$ such that, for all $n > N$ and $x \ge 0$, 
\begin{align}\label{ASAF1}
\big| \psin(x) - \psid(x) \big| \le \dfrac{C}{\sqrtn} \, .
\end{align}
Recall from \eqref{eq:psid} and \eqref{eq:gam} that $\psid(x) = e^{-\gam x}$, with $\gam = 2\tet \E Y \big/\E\big(Y^2\big)$.
\end{thm}

\begin{proof}
From Propositions \ref{prop:psid_scale} and \ref{prop:upsn} it follows that 
$$
\left(1 - \dfrac{\delta}{\sqrtn}\right) e^{-\gam x} < \psin(x) < e^{-\left(\gam - \frac{\alp}{\sqrtn} \right)x}.
$$
Subtracting $e^{-\gam x}$ from each side yields,
\begin{align}\notag
- \, \dfrac{\del}{\sqrtn} \, e^{-\gam x} < \psin(x) - e^{-\gam x} < e^{-\left(\gam - \frac{\alp}{\sqrtn}\right)x} - e^{-\gam x}.
\end{align}
Clearly, the left side is bounded below by $-\del/\sqrtn$. Basic calculus implies that, for every $n > \big(\alp/\gam\big)^2$, the right side is bounded above by 
$$
\left(1-\frac{\alp}{\gam\sqrtn}\right)^{\frac{\gam\sqrtn}{\alp}}\left(\frac{\frac{\alp}{\sqrtn}}{{\gam-\frac{\alp}{\sqrtn}}}\right).
$$
The first factor converges to $e^{-1}$, and the second factor is of order $\mO\big(n^{-1/2} \big)$. By combining this upper bound with the lower bound, we deduce inequality \eqref{ASAF1}.
\end{proof}

\begin{remark}
Theorem {\rm \ref{thm:psin_lim}} asserts that the rate of convergence of $\psi_n$ to $\psid$ is of order $\mO\big(n^{-1/2} \big)$, and, moreover, that the convergences is uniform over $x \in [0,\infty)$.  By using probabilistic techniques and relying on convergence in distribution of the underlying processes, others prove the pointwise convergence $\lim_{n \to \infty}\psin(x) = \psid(x)$ without estimating the rate.  The first to do so in the actuarial literature is Iglehart {\rm \cite{I1969};} for more recent work in this vein, see B\"auerle {\rm \cite{B2004}}.  \qed
\end{remark}

\begin{remark}\label{rem:lim}
From the proof of Theorem {\rm \ref{thm:psin_lim}}, relative to the limit $e^{-\gam x}$, we see that the relative error between $\psin$ and $e^{-\gam x}$ is bounded as follows$:$
\[
- \, \dfrac{\del}{\sqrtn} < \dfrac{\psin(x) - e^{-\gam x}}{e^{-\gam x}} < e^{\frac{\alp}{\sqrtn}x} - 1.
\]
Both lower and upper bounds are of order $\mO\big(n^{-1/2}\big)$, but the upper bound is not uniform in $x$.

That said, consider the Cram\'er-Lundberg asymptotic formula; see, for example, Theorem $5.7$ in Schmidli {\rm \cite{S2017}:}
\begin{equation}\label{eq:CLasym}
\lim_{x \to \infty} \psin(x) e^{\Rn x} = \dfrac{\frac{\tet}{\sqrtn} \E \Big( \frac{Y}{\sqrtn} \Big)}{\E \Big( \frac{Y}{\sqrtn} \, e^{\frac{\Rn Y}{\sqrtn}}\Big) -  \left( 1 + \frac{\tet}{\sqrtn} \right) \E \Big( \frac{Y}{\sqrtn} \Big)} = \dfrac{\tet \E Y}{\sqrtn \, \E \Big( Y e^{\frac{\Rn Y}{\sqrtn}}\Big) -  \left( \sqrtn + \tet \right) \E Y} \, ,
\end{equation}
in which $\Rn$ is the adjustment coefficient, that is, $\Rn$ is the positive root of
\[
0 = \E \Big( e^{\frac{rY}{\sqrtn}} \Big) - \left\{ 1 + \left( 1 + \dfrac{\tet}{\sqrtn} \right) \E \bigg( \dfrac{Y}{\sqrtn} \bigg) r \right\}.
\]
One can show that $\lim \limits_{n \to \infty} \Rn = \gam$, and
\begin{equation}\label{eq:CLlimit}
\lim_{n \to \infty} \lim_{x \to \infty} \psin(x) e^{\Rn x} = 1.
\end{equation}
See Appendix \ref{app:B} for a proof of these two limits.  Furthermore, from Theorem {\rm \ref{thm:psin_lim}}, we know that $\lim \limits_{n \to \infty} \psin(x) e^{\Rn x} = 1$ for all $x \ge 0$.  
\qed
\end{remark}

When $Y$ is exponentially distributed as in Example \ref{ex:exp}, we can strengthen the result of Theorem \ref{thm:psin_lim}.  Define the function $f:[0,\infty) \times [0, \tet] \to \R$ by 
$$
f(z,w) = \frac{1}{1+w} \exp\left(- \frac{z}{1+w} \right).
$$
Define $\eps_n=n^{-1/2}$; then, for any $n \in \N$ and $x \in [0, \infty)$, 
$$
\psin(x) = f (\tet \bet x, \tet \eps_n).
$$
Finally, define the function $g:[0,\infty) \times [0,1] \to \R$ by $g(x, y) = f(\tet \bet x, \tet y)$.  By using the function $f$, we provide, for the exponential case, a uniform asymptotic approximation for $\psin$ of any arbitrary order.

\begin{thm}\label{thm:exp}
If $Y \sim Exp(\bet)$ with mean $1/\bet$, then for any $k \in \N$, there exists $C = C(k) > 0$ such that, for all $n \in \N$ and $x \ge 0$, 
\begin{align}\label{ASAF2}
\left| \psin(x) - \sum_{m=0}^k\dfrac{\eps_n^m}{m!} \, \dfrac{\partial^m}{\partial y^m}g(x,0) \right| \le C\eps_n^{k+1} = \dfrac{C}{n^{(k+1)/2}} \, .
\end{align}
\end{thm}

\begin{proof}
Because $g$ linear transforms $f$, it is sufficient to show that 
\begin{align}\notag
\limsup_{w \to 0^+} \; \sup_{z \ge 0} \dfrac{1}{w^{k+1}} \left|
f(z, w) - \sum_{m=0}^k \dfrac{w^m}{m!}  \, \dfrac{\partial^m}{\partial w^m}f(z,0)
\right| < \infty.
\end{align}
By Taylor's expansion, for any $z \ge 0$ and $w \in [0, \tet]$, there exists $w_z \in [0,w]$ such that 
\begin{align}\notag
f(z,w) - \sum_{m=0}^k\dfrac{w^m}{m!} \, \dfrac{\partial^m}{\partial w^m}f(z,0) = \dfrac{w^{k+1}}{(k+1)!} \, \dfrac{\partial^{k+1}}{\partial w^{k+1}}f(z,w_z).  
\end{align}
Hence, the problem reduces to showing that 
\begin{align}\notag
\limsup_{w\to 0^+} \; \sup_{z \ge 0} \,
\left|
\dfrac{\partial^{k+1}}{\partial w^{k+1}}f(z,w_z)
\right| < \infty,
\end{align}
which in turn follows from the stronger inequality
\begin{align}\label{desired}
\sup_{w \in [0, \, \tet], \, z \ge 0} \,
\left|
\dfrac{\partial^{k+1}}{\partial w^{k+1}}f(z,w)
\right| < \infty.
\end{align}
One can show by induction that, for any $k \in \N$, there exists a bivariate polynomial $P_{k+1}$ such that 
\begin{align}\notag
\dfrac{\partial^{k+1}}{\partial w^{k+1}}f(z,w)=\frac{P_{k+1}(z,(1+w))}{(1+w)^{k+1}}f(z,w).
\end{align}
Since the domain of the variable $w$ is a compact set that is bounded away from $-1$, that is $[0,\theta]$, it follows that it is sufficient to show that for any $\ell \in \N$, 
\begin{align}\notag
\sup_{w \in [0, \, \tet], \, z \ge 0} \,
z^\ell f(z,w) < \infty,
\end{align}
which is clearly true. Hence, we obtain that \eqref{desired} holds, which finishes the proof of this theorem.
\end{proof}

\begin{example}\label{ex:gamma2}
Another way to think of \eqref{ASAF2} when $Y \sim Exp(\bet)$ and $k = 1$ is that the limit
\[
\lim_{n \to \infty} n \left( \psin(x) - e^{-\tet \bet x} \left\{ 1 + \dfrac{\tet}{\sqrtn} \, \big( \tet \bet x - 1 \big) \right\} \right)
\]
is finite for all $x \ge 0$ and bounded uniformly with respect to $x$.  However, when $Y \sim Gamma(2, \bet)$ as in Example {\rm \ref{ex:gamma}}, because of the discontinuity in the expansion at $x = 0$, the corresponding limit for $x > 0$ vanishingly small is approximately
\[
\lim_{n \to \infty} n \left( \dfrac{1}{1 + \frac{\tet}{\sqrtn}} - 1 + \dfrac{8\tet}{9\sqrtn} \right) = \lim_{n \to \infty} \left( - \, \dfrac{\tet \sqrtn}{9} \right) = - \infty.
\]
Thus, we cannot expect to do better than Theorem {\rm \ref{thm:psin_lim}} for a general claim severity $Y$.
\qed
\end{example}

\begin{remark}\label{rem:opt}
In light of the comment at the end of Example {\rm \ref{ex:gamma2}}, we motivate the $\mO\big(n^{-1/2} \big)$ rate of convergence.  Note that the pre-limit process weakly converges to a Brownian motion $($with drift$)$, and by the Skorokhod representation theorem, we can think about the convergence as uniform over compact time intervals. Now, because the pre-limit process has jumps of order $\mO\big(n^{-1/2} \big)$, it follows that at the first hitting time of $0$ of the Brownian motion, the pre-limit process is in an $\mO\big(n^{-1/2} \big)$-neighborhood of $0$.  Arguing by contradiction, assume for a moment that the rate in \eqref{ASAF1} can be improved to $o\big(n^{-1/2} \big)$, then the probability of ever hitting $0$, when starting at $c_1/\sqrtn$, is $e^{-\gam c_1/\sqrtn} + o\big(n^{-1/2} \big) \approx 1 + c_2/\sqrtn + o\big(n^{-1/2} \big)$, for some scalar $c_2 \in \R$.  The $c_2/\sqrtn$-term yields that the difference between the hitting probabilities is of order $\mO\big(n^{-1/2} \big)$, contradicting the improvement we conjectured.  \qed
\end{remark}

\section{Summary and future research}\label{sec:sum}

We proved a comparison lemma (Lemma \ref{lem:comp}) for the integro-differential equation that determines the probability of ruin for the Cram\'er-Lundberg (CL) model.  By using that comparison lemma, we showed, in Theorem \ref{thm:psin_lim}, that the rate of convergence of $\lim_{n \to \infty} \psin = \psid$ is of order $\mO\big(n^{-1/2}\big)$ and is {\it uniform} in $x \ge 0$.  Generally, one cannot improve on this rate of convergence, which we demonstrated in Example \ref{ex:gamma2} by example and discussed in Remark \ref{rem:opt}.  That said, for exponentially distributed claims, in Theorem \ref{thm:exp}, we showed that we can approximate $\psin$ up to any order, also uniformly in $x \ge 0$.

Many of the references in the bibliography also consider the finite-time ruin problem, which we did not address in this paper.  So, in future work we will find an asymptotic result for the probability of ruin in finite time, parallel to Theorem \ref{thm:psin_lim}.  More importantly, we will consider optimal control of the surplus process via reinsurance or optimal dividends.  Diffusion approximations (DA) are commonly applied to the surplus process before applying controls because the problem becomes tractable.  However, the optimal strategy in the CL case can be much different than the one obtained in the DA case.  For example, when minimizing the probability of ruin under the CL model, the optimal per-claim retention strategy for small values of surplus is to retain all of one's claims; by contrast, under the DA model, the optimal per-claim retention is strictly positive as surplus approaches $0$.  It would be interesting to see if we obtain an asymptotic result for the controlled probability of ruin.

Gerber, Shiu, and Smith \cite{GSS2008} addressed approximations to the dividend problem.  Also, B\"auerle \cite{B2004} considered the large-$\la$ approximation of the dividend problem and proved that, as the Poisson rate increases without bound with the corresponding scaling of claim severity as in this paper, the optimal value function converges to the one under the DA as $\la$ goes to infinity.  It would be interesting to determine if the rate of convergence of the optimal barrier is of order $\mO\big(n^{-1/2}\big)$, as suggested by the work in this paper.

Another line of research we will pursue is to improve the existing estimates for busy periods and sojourn times in queueing systems. Recall, from the introduction, that the diffusion approximation is often used in stochastic networks and is called the heavy-traffic approximation.  Also, integro-differential equations are commonly used to estimate expectations and probabilities, see, for example, \cite{T1976,PSZ,ABV}.  Hence, it would be natural to formulate the integro-differential for the {\it scaled} system and apply the method in this paper to attain convergence plus its rate for some magnitudes of interest.

\appendix

\section{$\Fn$ evaluated at $\psid$}\label{app:psid}

\setcounter{equation}{0}
\renewcommand{\theequation}
{A.\arabic{equation}}

In this appendix, we present the calculations that inspired Propositions \ref{prop:psid_scale} and \ref{prop:upsn}.

\begin{align}\label{eq:Fn_psid}
& \Fn \big( x, \psid(x), (\psid)_x(x), \psid(\cdot) \big) \notag \\
&\quad =  \la e^{- \gam x} \left\{ \big( \sqrtn + \tet \big) \E Y \gam - n \int_0^\infty \left( e^{\frac{\gam t}{\sqrtn}} - 1 \right) dF_Y(t) \right\} + n \la \int_{\sqrtn x}^\infty \left( e^{\gam \left(\frac{t}{\sqrtn} - x \right)} - 1 \right) dF_Y(t) \notag \\
&\quad =  \la e^{- \gam x} \left\{ \big( \sqrtn + \tet \big) \E Y \gam - n \int_0^\infty \left( \dfrac{\gam t}{\sqrtn} + \dfrac{\gam^2 t^2}{2n} \right) dF_Y(t) \right\} + n \la \int_{\sqrtn x}^\infty \left( e^{\gam \left(\frac{t}{\sqrtn} - x \right)} - 1 \right) dF_Y(t) \notag \\
&\qquad - n \la e^{- \gam x} \int_0^\infty \left( e^{\frac{\gam t}{\sqrtn}} - 1 - \dfrac{\gam t}{\sqrtn} - \dfrac{\gam^2 t^2}{2n} \right) dF_Y(t).
\end{align}
The terms in the curly brackets cancel, and we are left with
\begin{align*}
& \Fn \big( x, \psid(x), (\psid)_x(x), \psid(\cdot) \big) \\
&\quad = - n \la e^{- \gam x} \int_0^\infty \left( e^{\frac{\gam t}{\sqrtn}} - 1 - \dfrac{\gam t}{\sqrtn} - \dfrac{\gam^2 t^2}{2n} \right) dF_Y(t) + n \la \int_{\sqrtn x}^\infty \left( e^{\gam \left(\frac{t}{\sqrtn} - x \right)} - 1 \right) dF_Y(t) \\
&\quad = - n \la e^{- \gam x} \int_0^\infty \dfrac{\gam^3 t^3}{2 n^{3/2}} \int_0^1 (1 - \ome)^2 e^{\frac{\gam t}{\sqrtn} \, \ome} d \ome dF_Y(t) + n \la \int_{\sqrtn x}^\infty \left( e^{\gam \left(\frac{t}{\sqrtn} - x \right)} - 1 \right) dF_Y(t) \\
&\quad = - \la e^{- \gam x} \, \dfrac{\gam^3}{2 \sqrtn} \int_0^1 (1 - \ome)^2 \, \E \Big( Y^3 e^{\frac{\gam \ome}{\sqrtn} \, Y} \Big) d \ome + n \la \int_{\sqrtn x}^\infty \left( e^{\gam \left(\frac{t}{\sqrtn} - x \right)} - 1 \right) dF_Y(t).
\end{align*}
The first term is negative and of order $\mO \big(n^{-1/2} \big)$.  If we set $d = \sqrtn x$, then the second term becomes
\begin{align*}
&n \la \int_{\sqrtn x}^\infty \left( e^{\gam \left(\frac{t}{\sqrtn} - x \right)} - 1 \right) dF_Y(t) = n \la \int_d^\infty \left( e^{\frac{\gam}{\sqrtn} (t - d)} - 1 \right) dF_Y(t) \\
&\quad = n \la \int_d^\infty \left( e^{\frac{\gam}{\sqrtn}(t - d)} - 1 \right) dF_Y(t) = n \la S_Y(d) \int_d^\infty  \dfrac{\gam(t - d)}{\sqrtn} \int_0^1 e^{\frac{\gam(t - d)}{\sqrtn} \, \ome} d \ome \, \dfrac{dF_Y(t)}{S_Y(d)} \\
&\quad = \sqrtn \, \gam \la S_Y(d) \int_0^1 \E \Big((Y - d) \, e^{\frac{\gam \ome}{\sqrtn} \, (Y - d)} \, \Big| \, Y > d \Big) d \ome,
\end{align*}
which is positive and of order $\mO \big(\sqrtn \, \big)$.

To obtain a lower bound for $\psin$, we modify $\psid$ so that the corresponding modified second term is negative, and that is the gist of Proposition \ref{prop:psid_scale}.  The scaling does not affect the negative sign of the first term, and it makes the second term negative.  Also, note that the scaling effectively subtracts a term of order $\mO \big( n^{-1/2} \big)$ from $\psid$.

To obtain an upper bound for $\psin$, we modify $\psid$ so that the corresponding modified first term is positive, and that is the gist of Proposition \ref{prop:upsn}.  The additional exponent of $\alp/\sqrtn$ does not affect the positive sign of the second term, and it makes the first term positive.  Also, note that the modification of $\psid$'s exponent effectively adds a term of order $\mO \big( n^{-1/2} \big)$ to $\psid$.

\section{Proof of two limits stated in Remark \ref{rem:lim}}\label{app:B}

\setcounter{equation}{0}
\renewcommand{\theequation}
{B.\arabic{equation}}

First, we will prove that
\begin{equation}\label{eq:Rn_lim}
\lim_{n \to \infty} \Rn = \gam,
\end{equation}
in which $\Rn$ is the positive zero of the function $f_n$ defined by
\[
f_n(r) = \E \Big( e^{\frac{rY}{\sqrtn}} \Big) - \left\{ 1 + \left( 1 + \dfrac{\tet}{\sqrtn} \right) \E \bigg( \dfrac{Y}{\sqrtn} \bigg) r \right\},
\]
and we assume that $\Rn > 0$ exists.  From pages 90-91 of Schmidli \cite{S2017}, we know that
\[
\Rn < \gam,
\]
for all $n \in \N$.

Let $\eps > 0$; then, there exists $N$ such that
\begin{equation}\label{eq:eps_bnd}
\dfrac{\gam^2}{\sqrt{N} \, \E\big( Y^2 \big)} \, \E \bigg( Y^3 e^{\frac{\gam Y}{\sqrt{N}}} \, \bigg) < \eps.
\end{equation}
Then, by expanding the exponential in $f_n$, we obtain
\[
0 = \E \bigg( e^{\frac{\Rn Y}{\sqrtn}} - 1 - \dfrac{\Rn Y}{\sqrtn} - \dfrac{\Rn^2 Y^2}{2n} \bigg) + \dfrac{\Rn}{n} \left\{ \dfrac{1}{2} \, \Rn \E \big( Y^2 \big) - \tet \E Y \right\},
\]
or equivalently,
\[
0 = \E \bigg( \dfrac{\Rn^3 Y^3}{2 n^{3/2}} \, \int_0^1 (1 - \ome)^2 e^{\frac{\Rn Y}{\sqrtn} \, \ome} d \ome \bigg)  +  \dfrac{\Rn}{n} \left\{ \dfrac{1}{2} \, \Rn \E \big( Y^2 \big) - \tet \E Y \right\},
\]
or
\[
0 = \dfrac{\Rn^2}{\sqrtn \, \E \big( Y^2 \big)} \, \int_0^1 (1 - \ome)^2 \E \bigg( Y^3 e^{\frac{\Rn Y}{\sqrtn} \, \ome} \bigg) d \ome + (\Rn - \gam) .
\]
Then, \eqref{eq:eps_bnd} implies that, for $n > N$, we have
\[
0 < \eps + (\Rn - \gam),
\]
or
\[
0 < \gam - \Rn < \eps.
\]
Thus, we have proved the limit in \eqref{eq:Rn_lim}.

Next, we will prove \eqref{eq:CLlimit}.  Expand  the exponential in the denominator in the expression for $\lim \limits_{x \to \infty} \psin(x) e^{\Rn x}$ in \eqref{eq:CLasym}, specifically,
\begin{align*}
&\sqrtn \, \E \bigg( Y e^{\frac{\Rn Y}{\sqrtn}} \bigg) -  \left( \sqrtn + \tet \right) \E Y = \sqrtn \, \E \bigg( Y \bigg( e^{\frac{\Rn Y}{\sqrtn}} - 1 - \dfrac{\Rn Y}{\sqrtn} \bigg) \bigg)  + \Rn \E \big(Y^2 \big) - \tet \E Y \\
&= \sqrtn \, \E \bigg( Y \cdot \dfrac{\Rn^2 Y^2}{n} \, \int_0^1 (1 - \ome) e^{\frac{\Rn Y}{\sqrtn} \, \ome} d \ome \bigg) + \Rn \E \big(Y^2 \big) - \tet \E Y \\
&= \dfrac{\Rn^2}{\sqrtn} \int_0^1 (1 - \ome) \E \bigg( Y^3 e^{\frac{\Rn Y}{\sqrtn} \, \ome} \bigg) d \ome + (\Rn - \gam) \E \big(Y^2 \big) + \tet \E Y,
\end{align*}
and the limit of this denominator as $n$ goes to infinity equals $\tet \E Y$ because the first two terms go to $0$; thus, we have proved the limit in \eqref{eq:CLlimit}.

\vspace{\baselineskip}\noindent
\textbf{Acknowledgement:} The authors thank two anonymous referees for their suggestions that improved the presentation of the paper.

\end{document}